\documentclass[leqno,12pt]{article} %leqno is the option to put formula numbers on the left side
\usepackage{amsmath, amssymb}
\usepackage{amsthm} %theorem environment option
\theoremstyle{plain} %text of this environment is typesetted in italics
\newtheorem{theorem}{\indent\sc Theorem}[section]
\newtheorem{lemma}[theorem]{\indent\sc Lemma}

\newtheorem{proposition}[theorem]{\indent\sc Proposition}

\theoremstyle{definition} %text of this environment is typesetted in roman letters
\newtheorem{definition}[theorem]{\indent\sc Definition}
\newtheorem{remark}[theorem]{\indent\sc Remark}
\newtheorem{example}[theorem]{\indent\sc Example}

\title{On warped product gradient $\eta$-Ricci solitons} %title of the paper
\author{Adara M. Blaga}

\date{}
\begin{document}

\maketitle

\markboth{{\small\it {\hspace{8cm} Warped product gradient $\eta$-Ricci solitons}}}{\small\it{Warped product gradient $\eta$-Ricci solitons
\hspace{8cm}}}

%%%%%%%%%%%%%%% footnote %%%%%%%%%%%%%%%%
\footnote{ %2010 MSC numbers
}
\footnote{ %key words and phrases
}

\begin{abstract}
If the potential vector field of an $\eta$-Ricci soliton is of gradient type, using Bochner formula, we derive from the soliton equation a Laplacian equation satisfied by the potential function $f$. In a particular case of irrotational potential vector field we prove that the soliton is completely determined by $f$. We give a way to construct a gradient $\eta$-Ricci soliton on a warped product manifold and show that if the base manifold is oriented, compact and of constant scalar curvature, the soliton on the product manifold gives a lower bound for its scalar curvature.
\end{abstract}

\section{Introduction}

\textit{Ricci flow}, introduced by R. S. Hamilton \cite{ham2}, deforms a Riemannian metric $g$ by the evolution equation $\frac{\partial }{\partial t}g=-2S$, called the "heat equation" for Riemannian metrics, towards a canonical metric. Modeling the behavior of the Ricci flow near a singularity, \textit{Ricci solitons} \cite{ham} have been studied in the contexts of complex, contact and paracontact geometries \cite{be}.

The more general notion of \textit{$\eta$-Ricci soliton} was introduced by J. T. Cho and M. Kimura \cite{ch} and was treated by
C. C\u alin and M. Crasmareanu on Hopf hypersurfaces in complex space forms \cite{cacr}. We also discussed some aspects of $\eta$-Ricci solitons in paracontact \cite{bl}, \cite{blag} and Lorentzian para-Sasakian geometry \cite{bla}.

A particular case of soliton arises when the potential vector field is the gradient of a smooth function. The gradient vector fields play a central r\^{o}le in the Morse-Smale theory \cite{ms}. G. Y. Perelman
showed that if the manifold is compact, then the Ricci soliton is gradient \cite{pe}. In \cite{ha}, R. S. Hamilton
conjectured that a compact gradient Ricci soliton on a manifold $M$ with positive curvature operator
implies that $M$ is Einstein manifold. In \cite{de}, S. Deshmukh proved that a Ricci soliton of positive
Ricci curvature and whose potential vector field is of Jacobi-type, is compact and therefore, a gradient Ricci soliton.
Different aspects of gradient Ricci solitons
were studied in various papers. In \cite{ba}, N. Basu and A. Bhattacharyya treated gradient
Ricci solitons in Kenmotsu manifolds having Killing potential vector field. P. Petersen and W. Wylie discussed the rigidity of
gradient Ricci solitons \cite{pet2} and gave a classification imposing different curvature conditions \cite{pet1}.

The aim of our paper is to investigate some properties of gradient $\eta$-Ricci solitons. After deducing some results derived from the Bochner formula, we construct a gradient $\eta$-Ricci soliton on a warped product manifold and for the particular case of product manifolds, we show that if the base is oriented, compact and of constant scalar curvature, then we obtain a lower bound for the scalar curvature of the product manifold.

\section{Bochner formula revisited}

Let $(M,g)$ be an $m$-dimensional Riemannian manifold and consider $\xi$ a gradient vector field on $M$.
If $\xi:=grad{(f)}$, for $f$ a smooth function on $M$, then the $g$-dual $1$-form $\eta$ of $\xi$ is closed, as $\eta(X):=g(X,\xi)=df(X)$. Then $0=(d\eta)(X,Y):=X(\eta(Y))-Y(\eta(X))-\eta([X,Y])=g(\nabla_X\xi,Y)-g(\nabla_Y\xi, X)$, hence:
\begin{equation}\label{e41}
g(\nabla_X\xi,Y)=g(\nabla_Y\xi, X),
\end{equation}
for any $X$, $Y\in \chi(M)$, where $\nabla$ is the Levi-Civita connection of $g$.

Also:
\begin{equation}\label{e42}
div(\xi)=\Delta(f)
\end{equation}
and
\begin{equation}\label{e43}
div(\eta):=trace(Z\mapsto \sharp((\nabla\eta)(Z,\cdot)))=\sum_{i=1}^{m}(\nabla_{E_i}\eta)E_i=
\sum_{i=1}^{m}g(E_i,\nabla_{E_i}\xi):=div(\xi),
\end{equation}
for $\{E_i\}_{1\leq i\leq m}$ a local orthonormal frame field with $\nabla_{E_i}E_j=0$ in a point. From now on,
whenever we make a local computation, we will consider this frame.\\

In this case, the Bochner formula becomes:
\begin{equation}
\frac{1}{2}\Delta(|\xi|^2)=|\nabla \xi|^2+S(\xi,\xi)+\xi(div(\xi)),
\end{equation}
where $S$ is the Ricci curvature of $g$. Indeed:
\begin{equation}\label{e46}
(div(\mathcal{L}_{\xi}g))(X):=trace(Z\mapsto \sharp((\nabla(\mathcal{L}_{\xi}g))(Z,\cdot,X)))=\sum_{i=1}^{m}(\nabla_{E_i}(\mathcal{L}_{\xi}g))(E_i,X)=
\end{equation}
$$=\sum_{i=1}^{m}
\{E_i((\mathcal{L}_{\xi}g)(E_i,X))-(\mathcal{L}_{\xi}g)(E_i,\nabla_{E_i}X)\}=
2\sum_{i=1}^{m}g(\nabla_{E_i}\nabla_X\xi-\nabla_{\nabla_{E_i}X}\xi,E_i):=$$
$$:=2\sum_{i=1}^{m}g(\nabla^2_{E_i,X}\xi,E_i)=2\sum_{i=1}^{m}g(\nabla^2_{X,E_i}\xi+R(E_i,X)\xi,E_i)
:=$$$$:=2\sum_{i=1}^{m}g(\nabla^2_{X,E_i}\xi,E_i)+2trace(Z\mapsto R(Z,X)\xi):=
2\sum_{i=1}^{m}g(\nabla_{X}\nabla_{E_i}\xi-\nabla_{\nabla_X{E_i}}\xi,E_i)+2S(X,\xi)=$$$$=
2\sum_{i=1}^{m}g(\nabla_{X}\nabla_{E_i}\xi, E_i)+2S(X,\xi)=2\sum_{i=1}^{m}X(g(\nabla_{E_i}\xi, E_i))+2S(X,\xi)=2X(div(\xi))+2S(X,\xi),$$
where $R$ is the Riemann curvature and $S$ is the Ricci curvature tensor fields of the metric $g$ and the relation (\ref{e46}), for $X:=\xi$, becomes:
\begin{equation}\label{e49}
(div(\mathcal{L}_{\xi}g))(\xi)=2\xi(div(\xi))+2S(\xi,\xi).
\end{equation}

But the Bochner formula states that for any vector field $X$ \cite{pet2}:
\begin{equation}\label{e18}
(div(\mathcal{L}_{X}g))(X)=\frac{1}{2}\Delta(|X|^2)-|\nabla X|^2+S(X,X)+X(div(X))
\end{equation}
and from (\ref{e49}) and (\ref{e18}) we deduce that:
\begin{equation}\label{e52}
\Delta(|\xi|^2)-2|\nabla \xi|^2=2S(\xi,\xi)+2\xi(div(\xi)).
\end{equation}

Remark that (\ref{e46}) can be written in terms of $(1,1)$-tensor fields:
\begin{equation}
div (L_{\xi}g)=2d(div(\xi))+2i_{Q\xi}g,
\end{equation}
where $Q$ is the Ricci operator defined by $g(QX,Y):=S(X,Y)$.

\section{Gradient $\eta$-Ricci solitons}

Consider now the equation:
\begin{equation}\label{e11}
\mathcal{L}_{\xi}g+2S+2\lambda g+2\mu\eta\otimes \eta=0,
\end{equation}
where $g$ is a Riemannian metric, $S$ its Ricci curvature, $\eta$ a $1$-form and $\lambda$ and $\mu$ are real constants. The data $(g,\xi,\lambda,\mu)$ which satisfy
the equation (\ref{e11}) is said to be an \textit{$\eta$-Ricci soliton} on $M$ \cite{ch}; in particular, if $\mu=0$, $(g,\xi,\lambda)$ is a \textit{Ricci soliton} \cite{ham}. If the potential vector field $\xi$ is of gradient type, $\xi=grad(f)$, for $f$ a smooth function on $M$, then $(g,\xi,\lambda,\mu)$ is called \textit{gradient $\eta$-Ricci soliton}.

\begin{proposition}
Let $(M,g)$ be a Riemannian manifold. If (\ref{e11}) defines a gradient $\eta$-Ricci soliton on $M$
with the potential vector field $\xi:=grad(f)$ and $\eta$ is the $g$-dual $1$-form of $\xi$, then:
\begin{equation}
(\nabla_XQ)Y-(\nabla_YQ)X=-\nabla^2_{X,Y}\xi+\nabla^2_{Y,X}\xi+\mu(df\otimes \nabla\xi-\nabla\xi\otimes df)(X,Y),
\end{equation}
for any $X$, $Y\in\chi(M)$, where $Q$ stands for the Ricci operator.
\end{proposition}
\begin{proof}
As $g(QX,Y):=S(X,Y)$, follows:
\begin{equation}\label{e12}
\nabla\xi+Q+\lambda I_{\chi(M)}+\mu df\otimes \xi=0.
\end{equation}

Then:
\begin{equation}
(\nabla_XQ)Y=-(\nabla_X\nabla_Y\xi-\nabla_{\nabla_XY}\xi)-\mu\{g(Y,\nabla_X\xi)\xi+df(Y)\nabla_X\xi\}:=$$$$:=
-\nabla^2_{X,Y}\xi-\mu\{g(Y,\nabla_X\xi)\xi+df(Y)\nabla_X\xi\}
\end{equation}
and using (\ref{e41}) we get the required relation.
\end{proof}

\begin{theorem}\label{t}
If (\ref{e11}) defines a gradient $\eta$-Ricci soliton on the $m$-dimensional Riemannian manifold $(M,g)$ and
$\eta$ is the $g$-dual $1$-form of the gradient vector field $\xi:=grad(f)$, then:
\begin{equation}\label{e53}
\frac{1}{2}(\Delta-\nabla_{\xi})(|\xi|^2)=|Hess(f)|^2+\lambda |\xi|^2+\mu |\xi|^2\{|\xi|^2-2\Delta(f)\}.
\end{equation}
\end{theorem}
\begin{proof}
First remark that if $\xi=\sum _{i=1}^m\xi^iE_i$, for $\{E_i\}_{1\leq i\leq m}$ a local orthonormal frame field with $\nabla_{E_i}E_j=0$ in a point, then:
\begin{equation}
trace(\eta\otimes \eta)=\sum _{i=1}^m[df(E_i)]^2=\sum_{1\leq i,j,k\leq m}\xi^j\xi^k g(E_i,E_j)g(E_i,E_k)=\sum_{i=1}^m(\xi^i)^2=$$$$=\sum_{1\leq i,j\leq m}\xi^i\xi^j g(E_i,E_j)=|\xi|^2.
\end{equation}

Taking the trace of the equation (\ref{e11}), we obtain:
\begin{equation}\label{e13}
div(\xi)+scal+m\lambda +\mu |\xi|^2=0
\end{equation}
and differentiating it:
\begin{equation}\label{e14}
d(div(\xi))+d(scal)+\mu d(|\xi|^2)=0.
\end{equation}

Then taking the divergence of the same equation, we get:
\begin{equation}\label{e15}
div(\mathcal{L}_{\xi}g)+2div(S)+2\mu \cdot div(df\otimes df)=0.
\end{equation}

Substracting the relations (\ref{e15}) and (\ref{e14}) computed in $\xi$, considering (\ref{e49}),
(\ref{e52}) and using the fact that the scalar and the Ricci curvatures satisfy \cite{pet2}:
\begin{equation}\label{e19}
d(scal)=2div(S),
\end{equation}
we obtain:
\begin{equation}\label{e17}
\frac{1}{2}\Delta(|\xi|^2)-|\nabla \xi|^2+S(\xi,\xi)+\mu\{2(div(df\otimes df))(\xi)-\xi(|\xi|^2)\}=0.
\end{equation}

As
\begin{equation}\label{e20}
(div(df\otimes df))(\xi):=\sum_{i=1}^{m}\{E_i(df(E_i)df(\xi))-df(E_i)df(\nabla_{E_i}\xi)\}=$$$$=
\sum_{i=1}^{m}\{g(E_i,\xi)g(\nabla_{E_i}\xi,\xi)+g(\xi,\xi)g(E_i,\nabla_{E_i}\xi)\}
=g(\nabla_{\xi}\xi,\xi)+|\xi|^2\sum_{i=1}^{m}g(\nabla_{E_i}\xi,E_i)
:=$$$$:=\frac{1}{2}\xi(|\xi|^2)+|\xi|^2div(\xi),
\end{equation}
the equation (\ref{e17}) becomes:
\begin{equation}\label{e21}
\frac{1}{2}\Delta(|\xi|^2)-|\nabla \xi|^2+S(\xi,\xi)+2\mu |\xi|^2div(\xi)=0.
\end{equation}

From the $\eta$-soliton equation (\ref{e11}), we get:
\begin{equation}\label{e27}
S(\xi,\xi)=-\frac{1}{2}\xi(|\xi|^2)-\lambda |\xi|^2-\mu|\xi|^4,
\end{equation}
and the equation (\ref{e21}) becomes:
\begin{equation}\label{e28}
\frac{1}{2}\Delta(|\xi|^2)=|\nabla \xi|^2+\frac{1}{2}\xi(|\xi|^2)+\lambda |\xi|^2+
\mu |\xi|^4-2\mu |\xi|^2div(\xi).
\end{equation}

As $\xi:=grad(f)$ follows $Hess(f)=\nabla(df)$ and $|\nabla \xi|^2=|Hess(f)|^2$.
\end{proof}

\begin{remark}
For $\mu=0$ in Theorem \ref{t}, we obtain the relation for the particular case of gradient Ricci soliton \cite{pet2}.
\end{remark}

\begin{remark}
i) Assume that $\mu\neq 0$. Denoting by $\Delta_{\xi}:=\Delta-\nabla_{\xi}$, the equation (\ref{e53}) can be written:
$$\frac{1}{2}\Delta_{\xi}(|\xi|^2)=|Hess(f)|^2+|\xi|^2\{\lambda +\mu [|\xi|^2-2\Delta(f)]\},$$
where $\xi:=grad(f)$. If $\lambda \geq \mu [2\Delta(f)-|\xi|^2]$, then $\Delta_{\xi}(|\xi|^2)\geq 0$ and from the maximum principle follows that $|\xi|^2$ is constant in a neighborhood of any local maximum. If $|\xi|$ achieve its maximum, then $M$ is quasi-Einstein. Indeed, since $Hess(f)=0$, from (\ref{e11}) we have $S=-\lambda g-\mu df\otimes df$. Moreover, in this case, $|\xi|^2\{\lambda +\mu [|\xi|^2-2\Delta(f)]=0$, which implies either $\xi=0$, so $M$ is Einstein, or $|\xi|^2=2\Delta(f)-\frac{\lambda}{\mu}\geq 0$. Since $\Delta(f)=-scal-m\lambda-\mu |\xi|^2$ we get $\mu(2\mu+1)|\xi|^2=-(2\mu \cdot scal+2m\lambda\mu+\lambda)$. If $\mu=-\frac{1}{2}$, the scalar curvature equals to $\lambda(1-m)$ and if $\mu\neq-\frac{1}{2}$, it is either locally upper (or lower) bounded by $-\frac{\lambda(1+2m\mu)}{2\mu}$, for $\mu<-\frac{1}{2}$ ($\mu>-\frac{1}{2}$, respectively). On the other hand, if the potential vector field is of constant length, then $2\mu \Delta(f)\geq \lambda+\mu|\xi|^2$ equivalent to $\mu(2\mu+1)|\xi|^2+(2\mu \cdot scal+2m\lambda\mu+\lambda)\leq 0$ with equality for $\Delta(f)=\frac{\lambda}{2\mu}+\frac{|\xi|^2}{2}\geq\frac{\lambda}{2\mu}$ and $Hess(f)=0$ which yields the quasi-Einstein case.

ii) For $\mu=0$, we get the Ricci soliton case \cite{pet2}.
\end{remark}

\begin{proposition}\label{d}
Let $(M,g)$ be an $m$-dimensional Riemannian manifold and
$\eta$ be the $g$-dual $1$-form of the gradient vector field $\xi:=grad(f)$. If $\xi$ satisfies $\nabla \xi=I_{\chi(M)}-\eta\otimes\xi$, where $\nabla$ is the Levi-Civita connection associated to $g$, then:
\begin{enumerate}
  \item $Hess(f)=g-\eta\otimes \eta$;
  \item $R(X,Y)\xi=\eta(X)Y-\eta(Y)X$, for any $X$, $Y\in \chi(M)$;
  \item $S(\xi,\xi)=(1-m)|\xi|^2$.
\end{enumerate}
\end{proposition}
\begin{proof}
\begin{enumerate}
  \item Express the Lie derivative along $\xi$ as follows:
$$2(Hess(f))(X,Y)=(\mathcal{L}_{\xi}g)(X,Y):=\xi(g(X,Y))-g([\xi,X],Y)-g(X,[\xi,Y])=$$$$=\xi(g(X,Y))-g(\nabla_{\xi}X,Y)+g(\nabla_X\xi,Y)-
g(X,\nabla_{\xi}Y)+g(X,\nabla_Y\xi)=$$$$=g(\nabla_X\xi,Y)+g(X,\nabla_Y\xi)=2[g(X,Y)-\eta(X)\eta(Y)].$$
  \item Replacing now the expression of $\nabla \xi$ in $R(X,Y)\xi:=\nabla_X\nabla_Y\xi-\nabla_Y\nabla_X\xi-\nabla_{[X,Y]}\xi$, from a direct computation we get $R(X,Y)\xi=\eta(X)Y-\eta(Y)X$.
  \item $$S(\xi,\xi):=\sum_{i=1}^{m}g(R(E_i,\xi)\xi,E_i)=\sum_{i=1}^{m}\{[\eta(E_i)]^2-\eta(\xi)\}=
|\xi|^2-m|\xi|^2.$$
\end{enumerate}
\end{proof}

The condition satisfied by the potential vector field $\xi$, namely, $\nabla \xi=I_{\chi(M)}-\eta\otimes\xi$, naturally arises if $(M,\varphi, \xi,\eta,g)$ is for example, Kenmotsu manifold \cite{ke}. In this case, $M$ is a quasi-Einstein manifold.

\begin{example}\label{exa}
Let $M=\{(x,y,z)\in\mathbb{R}^3, z> 0\}$, where $(x,y,z)$ are the standard coordinates in $\mathbb{R}^3$. Set
$$\varphi:=-\frac{\partial}{\partial y}\otimes dx+ \frac{\partial}{\partial x}\otimes dy, \ \ \xi:=-z\frac{\partial}{\partial z}, \ \ \eta:=-\frac{1}{z}dz,$$
$$g:=\frac{1}{z^2}(dx\otimes dx+dy\otimes dy+dz\otimes dz).$$
Then $(\varphi , \xi , \eta , g)$ is a Kenmotsu structure on $M$.

Consider the linearly independent system of vector fields:
$$E_1:=z\frac{\partial}{\partial x}, \ \ E_2:=z\frac{\partial}{\partial y}, \ \ E_3:=-z\frac{\partial}{\partial z}.$$
Follows
$$\varphi E_1=-E_2, \ \ \varphi E_2=E_1, \ \ \varphi E_3=0,$$
$$\eta(E_1)=0, \ \ \eta(E_2)=0, \ \ \eta(E_3)=1,$$
$$[E_1,E_2]=0, \ \ [E_2,E_3]=E_2, \ \ [E_3,E_1]=-E_1$$
and the Levi-Civita connection $\nabla$ is deduced from Koszul's formula
$$2g(\nabla_XY,Z)=X(g(Y,Z))+Y(g(Z,X))-Z(g(X,Y))-$$$$-g(X,[Y,Z])+g(Y,[Z,X])+g(Z,[X,Y]),$$
precisely
$$\nabla_{E_1}E_1=-E_3, \ \ \nabla_{E_1}E_2=0, \ \ \nabla_{E_1}E_3=E_1,$$
$$\nabla_{E_2}E_1=0, \ \ \nabla_{E_2}E_2=-E_3, \ \ \nabla_{E_2}E_3=E_2,$$
$$\nabla_{E_3}E_1=0, \ \ \nabla_{E_3}E_2=0, \ \ \nabla_{E_3}E_3=0.$$
Then the Riemann and the Ricci curvature tensor fields are given by:
$$R(E_1,E_2)E_2=-E_1, \ \ R(E_1,E_3)E_3=-E_1, \ \ R(E_2,E_1)E_1=-E_2,$$ $$R(E_2,E_3)E_3=-E_2, \ \ R(E_3,E_1)E_1=-E_3, \ \ R(E_3,E_2)E_2=-E_3,$$
$$S(E_1,E_1)=S(E_2,E_2)=S(E_3,E_3)=-2.$$
From (\ref{e11}) computed in $(E_i,E_i)$:
$$2[g(E_i,E_i)-\eta(E_i)\eta(E_i)]+2S(E_i,E_i)+2\lambda g(E_i,E_i)+2\mu\eta(E_i)\eta(E_i)=0,$$
for all $i\in\{1,2,3\}$, we have:
$$2(1-\delta_{i3})-4+2\lambda+2\mu\delta_{i3}=0 \ \ \ \Longleftrightarrow \ \ \ \lambda-1+(\mu-1)\delta_{i3}=0,$$
for all $i\in\{1,2,3\}$. Therefore, $\lambda=\mu=1$ define an $\eta$-Ricci soliton on $(M, \varphi , \xi , \eta , g)$. Moreover, it is a gradient $\eta$-Ricci soliton, as the potential vector field $\xi$ is of gradient type, $\xi=grad(f)$, where $f(x,y,z):=-\ln z$.
\end{example}

\bigskip

Assume now that (\ref{e11}) defines a gradient $\eta$-Ricci soliton on $(M,g)$ with $\mu\neq 0$. Under the hypotheses of the Proposition \ref{d}, the equation (\ref{e28}) simplifies a lot. Compute:
\begin{equation}\label{e39}
|\nabla \xi|^2:=\sum_{i=1}^{m}g(\nabla_{E_i}\xi,\nabla_{E_i}\xi)=
\sum_{i=1}^{m}\{1+(|\xi|^2-2)[\eta(E_i)]^2\}=m+|\xi|^2(|\xi|^2-2),
\end{equation}
for $\{E_i\}_{1\leq i\leq m}$ a local orthonormal frame field with $\nabla_{E_i}E_j=0$ in a point,
\begin{equation}\label{e54}
\xi(|\xi|^2)=\xi(g(\xi,\xi))=2g(\nabla_{\xi}\xi,\xi)=2(|\xi|^2-|\xi|^4),
\end{equation}
\begin{equation}\label{e55}
\xi(|\xi|^4)=2|\xi|^2\xi(|\xi|^2)=4(|\xi|^4-|\xi|^6).
\end{equation}

From the equation (\ref{e11}) we obtain:
\begin{equation}\label{777}
S(\xi,\xi)=-(\lambda+1)|\xi|^2-(\mu-1)|\xi|^4.
\end{equation}

Using Proposition \ref{d} and the relation (\ref{777}), we get:
\begin{equation}
|\xi|^2=(m-1-\lambda)|\xi|^2-(\mu-1)|\xi|^4,
\end{equation}
so $|\xi|^2(\mu-1)=m-2-\lambda$ i.e. $\xi$ is of constant length. Using (\ref{e54}) we get $|\xi|=1$. It follows $\lambda+\mu=m-1$ and we deduce:

\begin{theorem}
Under the hypotheses of the Proposition \ref{d}, if (\ref{e11}) defines a gradient $\eta$-Ricci soliton on $(M,g)$ with $\mu\neq 0$, then the Laplacian equation (\ref{e28}) becomes:
\begin{equation}\label{e16}
\Delta(f)=\frac{m-1}{\mu}.
\end{equation}
\end{theorem}

Therefore, the existence of a gradient $\eta$-Ricci soliton defined by (\ref{e11}) with the potential vector field $\xi:=grad(f)$, yields the Laplacian equation (\ref{e16}), and the soliton is completely determined by $f$.

\section{Warped product $\eta$-Ricci solitons}

Consider $(B,g_B)$ and $(F,g_F)$ two Riemannian manifolds of dimensions $n$ and $m$, respectively. Denote by $\pi$ and $\sigma$ the projection maps from the product manifold $B\times F$ to $B$ and $F$ and by $\widetilde{\varphi}:=\varphi \circ \pi$ the lift to $B\times F$ of a smooth function $\varphi$ on $B$.
In this context, we shall call $B$ \textit{the base} and $F$ \textit{the fiber} of $B\times F$, the unique element $\widetilde{X}$ of $\chi(B\times F)$ that is $\pi$-related to $X\in \chi(B)$ and to the zero vector field on $F$, the \textit{horizontal lift of $X$} and the unique element $\widetilde{V}$ of $\chi(B\times F)$ that is $\sigma$-related to $V\in \chi(F)$ and to the zero vector field on $B$, the \textit{vertical lift of $V$}.
Also denote by $\mathcal{L}(B)$ the set of all horizontal lifts of vector fields on $B$, by $\mathcal{L}(F)$ the set of all vertical lifts of vector fields on $F$, by $\mathcal{H}$ the orthogonal projection of $T_{(p,q)}(B\times F)$ onto its horizontal subspace $T_{(p,q)}(B\times \{q\})$ and by
$\mathcal{V}$ the orthogonal projection of $T_{(p,q)}(B\times F)$ onto its vertical subspace $T_{(p,q)}(\{p\}\times F)$.

Let $\varphi>0$ be a smooth function on $B$ and
\begin{equation}\label{e7}
g:=\pi^* g_B+(\varphi\circ \pi)^2\sigma^*g_F
\end{equation}
be a Riemannian metric on $B\times F$.
\begin{definition}\cite{bi}
The product manifold of $B$ and $F$ together with the Riemannian metric $g$ defined by (\ref{e7}) is called the warped product of $B$ and $F$ by the warping function $\varphi$ (and is denoted by $(M:=B\times_{\varphi} F,g)$).
\end{definition}

Remark that if $\varphi$ is constant equal to 1, the warped product becomes the usual product of the Riemannian manifolds.

\bigskip

For simplification, in the rest of the paper we shall simply denote by $X$ the horizontal lift of $X\in \chi(B)$ and by $V$ the vertical lift of $V\in \chi(F)$.

\bigskip

Due to a result of J. Case, Y.-J. Shu and G. Wei \cite{case}, we know that for a gradient $\eta$-Ricci soliton $(g,\xi:=grad(f),\lambda,\mu)$ with $\mu\in (-\infty,0)$ and $\eta=df$ the $g$-dual of $\xi$, on a connected $n$-dimensional Riemannian manifold $(M,g)$, the function
\begin{equation}
e^{2\mu f}[\Delta(f)-|\xi|^2-\frac{\lambda}{\mu}]
\end{equation}
is constant.

Choosing properly an Einstein manifold, a smooth function and considering the warped product manifold, we can characterize the gradient $\eta$-Ricci soliton on the base manifold as follows \cite{case}. Let $(B,g_B)$ be an $n$-dimensional connected Riemannian manifold, $\lambda$ and $\mu$ real constants such that $-\frac{1}{\mu}$ is a natural number, $f$ a smooth function on $B$, $k:=\mu e^{2\mu f}[\Delta(f)-|\xi|^2-\frac{\lambda}{\mu}]$ and $(F,g_F)$ an $m$-dimensional Riemannian manifold with $m=-\frac{1}{\mu}$ and $S_F=kg_F$. Then $(g,\xi:=grad(f),\lambda,\mu)$ is a gradient $\eta$-Ricci soliton on $(B,g_B)$ with $\eta=df$ the $g$-dual of $\xi$, if and only if the warped product manifold $(M:=B\times_{\varphi}F,g)$ with the warping function $\varphi=e^{-\frac{f}{m}}$ is Einstein manifold with $S=\lambda g$.

\bigskip

Let
$S$, $S_B$, $S_F$ the Ricci tensors on $M$, $B$ and $F$ and $\widetilde{S_B}$, $\widetilde{S_F}$ the lift on $M$ of $S_B$ and $S_F$, which satisfy the following properties:

\begin{lemma}\cite{bi}\label{l}
If $(M:=B\times_{\varphi} F,g)$ is the warped product of $B$ and $F$ by the warping function $\varphi$ and $m>1$, then for any $X$, $Y\in \mathcal{L}(B)$ and any $V$, $W\in \mathcal{L}(F)$, we have:
\begin{enumerate}
  \item $S(X,Y)=\widetilde{S_B}(X,Y)-\frac{m}{\widetilde{\varphi}}H^{\varphi}(X,Y)$, where $H^{\varphi}$ is the lift on $M$ of $Hess(\varphi)$;
  \item $S(X,V)=0$;
  \item $S(V,W)=\widetilde{S_F}(V,W)-
      \pi^*[\frac{\Delta(\varphi)}{\varphi}+(m-1)\frac{|grad(\varphi)|^2}{\varphi^2}]|_Fg(V,W)$.
\end{enumerate}
\end{lemma}

Notice that the lift on $M$ of the gradient and the Hessian satisfy:
\begin{equation}
grad(\widetilde{f})=\widetilde{grad(f)},
\end{equation}
\begin{equation}
(Hess(\widetilde{f}))(X,Y)=\widetilde{(Hess(f))(X,Y)}, \ \ \textit{for any } \ X, Y \in \mathcal{L}(B).
\end{equation}
for any smooth function $f$ on $B$.

\bigskip

We shall construct a gradient $\eta$-Ricci soliton on a warped product manifold following \cite{fe}.

Let $(B,g_B)$ be a Riemannian manifold, $\varphi>0$ a smooth function on $B$ and $f$ a smooth function on $B$ such that:
\begin{equation}\label{e36}
S_B+Hess(f)-\frac{m}{\varphi}Hess(\varphi)+\lambda g_B+\mu df\otimes df=0,
\end{equation}
where $\lambda$, $\mu$
and $m>1$ are real constants.

Take $(F,g_F)$ an $m$-dimensional manifold with $S_F=k g_F$, for $k:=\pi^*[-\lambda \varphi^2+\varphi\Delta(\varphi)+(m-1)|grad(\varphi)|^2-\varphi(grad(f))(\varphi)]|_F$, where $\pi$ and $\sigma$ be the projection maps from the product manifold $B\times F$ to $B$ and $F$, respectively, and $g:=\pi^* g_B+(\varphi\circ \pi)^2\sigma^*g_F$ a Riemannian metric on $B\times F$. Then, for $\xi:=grad(f \circ \pi)$, if consider $\mu=0$ in (\ref{e36}), $(g, \xi, \lambda)$ is a gradient Ricci soliton on $B\times_{\varphi} F$ called the \textit{warped product Ricci soliton} \cite{fe}.

With the above notations, we prove that:
\begin{theorem}
Let $(B,g_B)$ be a Riemannian manifold, $\varphi>0$, $f$ two smooth functions on $B$, let $m>1$, $\lambda$, $\mu$ be real constants satisfying (\ref{e36}) and $(F,g_F)$ an $m$-dimensional Riemannian manifold. Then $(g,\xi,\lambda,\mu)$ is a gradient $\eta$-Ricci soliton on the warped product manifold $(B\times_{\varphi} F,g)$, where $\xi=grad(\widetilde{f})$ and the $1$-form $\eta$ is the $g$-dual of $\xi$,
if and only if:
\begin{equation}\label{e32}
S_B=-Hess(f)+\frac{m}{\varphi}Hess(\varphi)-\lambda g_B-\mu df\otimes df
\end{equation}
and
\begin{equation}\label{e44}
S_F=kg_F,
\end{equation}
where $k:=\pi^*[-\lambda \varphi^2+\varphi\Delta(\varphi)+(m-1)|grad(\varphi)|^2-\varphi(grad(f))(\varphi)]|_F$.
\end{theorem}
\begin{proof}
The gradient $\eta$-Ricci soliton $(g,\xi,\lambda,\mu)$ on $(B\times_{\varphi} F,g)$ is given by:
\begin{equation}\label{e31}
Hess(\widetilde{f})+S+\lambda g+\mu\eta\otimes \eta=0.
\end{equation}

Then for any $X$, $Y\in \mathcal{L}(B)$ and for any $V$, $W\in \mathcal{L}(F)$, from Lemma \ref{l} we get:
$$H^{f}(X,Y)+\widetilde{S_B}(X,Y)-\frac{m}{\widetilde{\varphi}}H^{\varphi}(X,Y)+\lambda g_B(X,Y)+\mu df(X)df(Y)=0$$
$$H^{f}(V,W)+\widetilde{S_F}(V,W)-\pi^*[\varphi\Delta(\varphi)+(m-1)|grad(\varphi)|^2-\lambda \varphi^2]|_Fg(V,W)=0$$
and using the fact that
$$H^{f}(V,W)=(Hess(\widetilde{f}))(V,W)=g(\nabla_V(grad(\widetilde{f})),W)
=\pi^*[\frac{(grad(f))(\varphi)}{\varphi}]|_F\widetilde{\varphi}^2g_F(V,W),$$
we obtain:
$$\widetilde{S_F}(V,W)=\pi^*[\varphi\Delta(\varphi)+(m-1)|grad(\varphi)|^2-\varphi(grad(f))(\varphi)-\lambda \varphi^2]|_Fg_F(V,W).$$

Conversely, notice that the left-hand side term in (\ref{e31}) computed in $(X,V)$, for $X\in \mathcal{L}(B)$ and $V\in \mathcal{L}(F)$ vanishes identically and using again Lemma \ref{l}, for each situation $(X,Y)$ and $(V,W)$, we can recover the equation (\ref{e31}) from (\ref{e32}) and (\ref{e44}).
\end{proof}

\begin{remark}\label{r1}
In the case of product manifold (for $\varphi =1$), notice that the equation (\ref{e36}) defines a gradient $\eta$-Ricci soliton on $B$ and the chosen manifold $(F, g_F)$ is Einstein ($S_F=-\lambda g_F$), so a gradient $\eta$-Ricci soliton on the product manifold $B\times F$ can be naturally obtained by "lifting" a gradient $\eta$-Ricci soliton on $B$.
\end{remark}

\begin{remark}
If for the function $\varphi$ and $f$ on $B$ there exists two constants $a$ and $b$ such that $\nabla (grad(\varphi))=\varphi [aI_{\chi(B)}+bdf\otimes grad(f)]$, then $Hess(\varphi)=\varphi (ag_B+bdf\otimes df)$ and $(g_B,grad(f),\lambda-ma,\mu-mb)$ is a gradient $\eta$-Ricci soliton on $B$.
\end{remark}

Let us make some remark on the class of manifolds that satisfy the condition (\ref{e36}):
\begin{equation}\label{e37}
S_B+Hess(f)-\frac{m}{\varphi}Hess(\varphi)+\lambda g_B+\mu df\otimes df=0,
\end{equation}
for $\varphi>0$, $f$ smooth functions on the oriented and compact Riemannian manifold $(B,g_B)$, $\lambda$, $\mu$
and $m>1$ real constants. Denote by $\xi:=grad(f)$.

Taking the trace of (\ref{e37}), we obtain:
\begin{equation}\label{e50}
scal_B+\Delta(f)-m\frac{\Delta(\varphi)}{\varphi}+n\lambda+\mu |\xi|^2=0.
\end{equation}

Remark that:
\begin{equation}\label{e47}
|Hess(f)-\frac{\Delta(f)}{n}g_B|^2:=\sum_{1\leq i,j\leq n}[Hess(f)(E_i,E_j)-\frac{\Delta(f)}{n}g_B(E_i,E_j)]^2=
\end{equation}
$$
=|Hess(f)|^2-2\frac{\Delta(f)}{n}\sum_{i=1}^ng_B(\nabla_{E_i}\xi,E_i)+\frac{(\Delta(f))^2}{n}=|Hess(f)|^2-\frac{(\Delta(f))^2}{n}.$$

Also:
$$(div(Hess(f)))(\xi):=\sum_{i=1}^n(\nabla_{E_i}(Hess(f)))(E_i,\xi)=$$$$=\sum_{i=1}^n[E_i(Hess(f)(E_i,\xi))-Hess(f)(E_i,\nabla_{E_i}\xi)]=$$
$$=\sum_{i=1}^nE_i(g_B(\nabla_{E_i}\xi,\xi))-\sum_{i=1}^ng_B(\nabla_{E_i}\xi,\nabla_{E_i}\xi)=
\sum_{i=1}^ng_B(\nabla_{E_i}\nabla_{\xi}\xi,E_i)-|\nabla \xi|^2:=$$$$:=div(\nabla_{\xi}\xi)-|Hess(f)|^2$$
and
$$div(\nabla_{\xi}\xi):=\sum_{i=1}^ng_B(\nabla_{E_i}\nabla_{\xi}\xi, E_i)=\sum_{i=1}^nE_i(g_B(\nabla_{\xi}\xi,E_i))=\sum_{i=1}^nE_i(Hess(f)(\xi,E_i))
=$$$$=\sum_{i=1}^n(\nabla_{E_i}(Hess(f)(\xi)))E_i:=div(Hess(f)(\xi)),$$
therefore:
\begin{equation}\label{e40}
(div(Hess(f)))(\xi)=div(Hess(f)(\xi))-|Hess(f)|^2.
\end{equation}

Applying the divergence to (\ref{e37}), computing it in $\xi$ and considering (\ref{e20}), we get:
\begin{equation}\label{e38}
(div(Hess(f)))(\xi)=-(div(S_B))(\xi)+m(div(\frac{Hess(\varphi)}{\varphi}))(\xi)-\mu (\frac{1}{2}d(|\xi|^2)+\Delta(f)df)(\xi)=
\end{equation}
$$=-\frac{d(scal_B)(\xi)}{2}+\frac{m}{\varphi}(div(Hess(\varphi)))(\xi)-\frac{m}{\varphi^2}Hess(\varphi)(grad(\varphi),\xi)-\mu [\frac{1}{2}d(|\xi|^2)(\xi)+\Delta(f)|\xi|^2]=$$
$$=-\frac{d(scal_B)(\xi)}{2}+m\cdot div(Hess(\varphi)(\frac{\xi}{\varphi}))-\frac{m}{\varphi}\langle Hess(f), Hess(\varphi)\rangle-\mu [\frac{1}{2}d(|\xi|^2)(\xi)+\Delta(f)|\xi|^2].$$

From (\ref{e50}), (\ref{e47}), (\ref{e40}) and (\ref{e38}), we obtain:
\begin{equation}\label{e51}
div(Hess(f)(\xi))=|Hess(f)-\frac{\Delta(f)}{n}g_B|^2-\frac{scal_B}{n}\Delta(f)+\frac{m}{n}\frac{\Delta(\varphi)}{\varphi}\Delta(f)-div(\lambda \xi)-
\end{equation}
$$-\frac{d(scal_B)(\xi)}{2}+m\cdot div(Hess(\varphi)(\frac{\xi}{\varphi}))-\frac{m}{\varphi}\langle Hess(f), Hess(\varphi)\rangle-\frac{\mu}{2} d(|\xi|^2)(\xi)- \frac{n+1}{n}\mu |\xi|^2 \Delta(f).$$

Integrating with respect to the canonical measure on $B$, we have:
$$\int_Bd(scal_B)(\xi)=\int_B\langle grad(scal_B), \xi\rangle=-\int_B\langle scal_B, div(\xi)\rangle=-\int_Bscal_B \cdot \Delta(f)$$
and similarly:
$$\int_Bd(|\xi|^2)(\xi)=\int_B\langle grad(|\xi|^2), \xi\rangle=-\int_B\langle |\xi|^2, div(\xi)\rangle=-\int_B|\xi|^2 \cdot \Delta(f).$$

Using:
$$|\xi|^2\cdot div(\xi)=div(|\xi|^2\xi)-|\xi|^2$$
and integrating (\ref{e51}) on $B$, from the above relations and the divergence theorem, we obtain:
\begin{equation}\label{e56}
\frac{n-2}{2n}\int_B \langle grad(scal_B), \xi\rangle=\int_B |Hess(f)-\frac{\Delta(f)}{n}g_B|^2-m\int_B \frac{1}{\varphi}\langle Hess(f), Hess(\varphi)\rangle+
\end{equation}
$$+\frac{m}{n}\int_B \frac{\Delta(\varphi)}{\varphi}\Delta(f)+\frac{n+2}{2n}\mu\int_B |\xi|^2.$$

\begin{proposition}
Let $(B,g_B)$ be an oriented and compact Riemannian manifold, $f$ a smooth function on $B$, let $m>1$, $\lambda$, $\mu$ be real constants satisfying (\ref{e36}) (for $\varphi=1$) and $(F,g_F)$ be an $m$-dimensional Einstein manifold with $S_F=-\lambda g_F$. If $(g,\xi,\lambda,\mu)$ is a gradient $\eta$-Ricci soliton on the product manifold $(B\times F,g)$, where $\xi=grad(\widetilde{f})$ and the $1$-form $\eta$ is the $g$-dual of $\xi$, then:
\begin{equation}\label{e57}
\frac{n-2}{2n}\int_B \langle grad(scal_B), \xi\rangle=\int_B |Hess(f)-\frac{\Delta(f)}{n}g_B|^2+\frac{n+2}{2n}\mu\int_B |\xi|^2.
\end{equation}
\end{proposition}

Let now consider the product manifold $B\times F$, in which case (\ref{e50}) (for $\varphi=1$) becomes:
\begin{equation}
scal_B+\Delta(f)+n\lambda+\mu |\xi|^2=0
\end{equation}
and integrating it on $B$, we get:
\begin{equation}
\mu\int_B |\xi|^2=-\int_B scal_B-n\lambda \cdot vol(B).
\end{equation}

Replacing it in (\ref{e57}), we obtain:
\begin{equation}\label{e6}
\frac{n-2}{2n}\int_B \langle grad(scal_B), \xi\rangle+\frac{n+2}{2n}\int_B scal_B=\int_B |Hess(f)-\frac{\Delta(f)}{n}g_B|^2-\frac{n+2}{2}\lambda \cdot vol(B).
\end{equation}

\begin{proposition}
Let $(B,g_B)$ be an oriented, compact and complete $n$-dimensional ($n>1$) Riemannian manifold of constant scalar curvature, $\varphi>0$, $f$ two smooth functions on $B$, let $m>1$, $\lambda$, $\mu$ be real constants satisfying (\ref{e37}). If one of the following two conditions hold:
\begin{enumerate}
  \item $\varphi=1$ and $\lambda=-\frac{scal_B}{n}$;
  \item there exists a positive function $h$ on $B$ such that $Hess(f)=-h\cdot Hess(\varphi)$ and $\mu\geq 0$,
\end{enumerate}
then $B$ is conformal to a sphere in the $(n+1)$-dimensional Euclidean space.
\end{proposition}
\begin{proof}
\begin{enumerate}
         \item From (\ref{e6}) we obtain:
$$\int_B |Hess(f)-\frac{\Delta(f)}{n}g_B|^2=\frac{n+2}{2}(\frac{scal_B}{n}+\lambda)\mu\cdot vol(B),$$
so $Hess(f)=\frac{\Delta(f)}{n}g_B$ which implies by \cite{ya} that $B$ is conformal to a sphere in the $(n+1)$-dimensional Euclidean space.
         \item From the condition $Hess(f)=-h\cdot Hess(\varphi)$ we obtain $\Delta(f)=-h \Delta(\varphi)$ and replacing them in (\ref{e56}), we get:
$$\int_B |Hess(f)-\frac{\Delta(f)}{n}g_B|^2+\frac{n+2}{2n}\mu \int_B |\xi|^2=0.$$

     From $\mu\geq0$ we deduce that $Hess(f)=\frac{\Delta(f)}{n}g_B$ and according to \cite{ya}, we get the conclusion.
      \end{enumerate}
\end{proof}

Finally, we state a result on the scalar curvature of a product manifold admitting an $\eta$-Ricci soliton:
\begin{proposition}
Let $(B,g_B)$ be an oriented and compact Riemannian manifold of constant scalar curvature, $f$ a smooth function on $B$, let $m>1$, $\lambda$, $\mu$ be real constants satisfying (\ref{e36}) (for $\varphi=1$) and $(F,g_F)$ be an $m$-dimensional Einstein manifold with $S_F=-\lambda g_F$. If $(g,\xi,\lambda,\mu)$ is a gradient $\eta$-Ricci soliton on the product manifold $(B\times F,g)$, where $\xi=grad(\widetilde{f})$ and the $1$-form $\eta$ is the $g$-dual of $\xi$, then the scalar curvature of $B\times F$ is $\geq -(n+m)\lambda$.
\end{proposition}
\begin{proof}
From (\ref{e6}) we deduce that $\frac{n+2}{2}(\frac{scal_B}{n}+\lambda)\cdot vol(B)=\int_B |Hess(f)-\frac{\Delta(f)}{n}g_B|^2\geq 0$ and since $scal_F=-m\lambda$, we get the conclusion.
\end{proof}

\bigskip

We end these considerations by giving an example of gradient $\eta$-Ricci soliton on a product manifold.

\begin{example}
Let $(g_M, \xi_M, 1,1)$ be the gradient $\eta$-Ricci soliton on the Riemannian manifold $M=\{(x,y,z)\in\mathbb{R}^3, z> 0\}$, where $(x,y,z)$ are the standard coordinates in $\mathbb{R}^3$, with the metric $g_M:=\frac{1}{z^2}(dx\otimes dx+dy\otimes dy+dz\otimes dz)$ (given by Example \ref{exa}) and let $S^3$ be the $3$-sphere with the round metric $g_S$ (which is Einstein with the Ricci tensor equals to $2g_S$). By Remark \ref{r1} we obtain the gradient $\eta$-Ricci soliton $(g, \xi, 1,1)$ on the "generalized cylinder" $M\times S^3$, where $g=g_M+g_S$ and $\xi$ is the lift on $M\times S^3$ of the gradient vector field $\xi_M=grad(f)$, where $f(x,y,z):=-\ln z$.
\end{example}

\small{

\bigskip

\textit{Adara M. Blaga}

\textit{Department of Mathematics}

\textit{West University of Timi\c{s}oara}

\textit{Bld. V. P\^{a}rvan nr. 4, 300223, Timi\c{s}oara, Rom\^{a}nia}

\textit{adarablaga@yahoo.com}
}

\end{document}